\documentclass[12pt,
psamsfonts]{amsart}

\usepackage{amssymb,amsfonts,amsmath}
\usepackage[all,arc]{xy}
\usepackage{enumerate}
\usepackage{mathrsfs}
\usepackage{fullpage}
\usepackage{xspace}
\usepackage[margin=1.0in]{geometry}
\usepackage{tcolorbox}
\usepackage{tikz-cd}
\usepackage{color}
\usepackage{aliascnt}
\usepackage[foot]{amsaddr}
\usepackage{hyperref}

\newtheorem{thm}{Theorem}[section]

\newaliascnt{theo}{thm}
\newtheorem{theo}[theo]{Theorem}
\aliascntresetthe{theo}

\newaliascnt{cor}{thm}
\newtheorem{cor}[cor]{Corollary}
\aliascntresetthe{cor}

\newaliascnt{prop}{thm}
\newtheorem{prop}[prop]{Proposition}
\aliascntresetthe{prop}

\newaliascnt{lem}{thm}
\newtheorem{lem}[lem]{Lemma}
\aliascntresetthe{lem}

\newaliascnt{conj}{thm}
\newtheorem{conj}[conj]{Conjecture}
\aliascntresetthe{conj}

\newaliascnt{que}{thm}

\aliascntresetthe{que}

\newaliascnt{notn}{thm}
\newtheorem{notn}[notn]{Notation}
\aliascntresetthe{notn}

\theoremstyle{remark}
\newaliascnt{rem}{thm}
\newtheorem{rem}[rem]{Remark}
\aliascntresetthe{rem}

\theoremstyle{assumption}
\newaliascnt{ass}{thm}

\aliascntresetthe{ass}

\theoremstyle{definition}
\newtheorem{defn}[thm]{Definition}

\newtheorem{exmp}[thm]{Example}

\newcommand{\Z}{\mathbb{Z}\xspace}
\newcommand{\Q}{\mathbb{Q}\xspace}

\DeclareMathOperator{\Spec}{Spec}
\DeclareMathOperator{\res}{res}
\DeclareMathOperator{\nd}{nd}

\DeclareMathOperator{\ord}{ord}

\DeclareMathOperator{\dv}{div}
\DeclareMathOperator{\alb}{alb}
\DeclareMathOperator{\img}{Im}

\DeclareMathOperator{\Pic}{Pic}
\DeclareMathOperator{\Gal}{Gal}

\DeclareMathOperator{\Alb}{Alb}
\DeclareMathOperator{\CH}{CH}
\DeclareMathOperator{\RS}{RS}
\DeclareMathOperator{\pr}{pr}
\DeclareMathOperator{\id}{id}
\DeclareMathOperator{\Ext}{Ext}
\DeclareMathOperator{\Ind}{Ind}
\makeatletter
\let\c@equation\c@thm
\makeatother
\numberwithin{equation}{section}

\title{Filtrations of the Chow group of zero-cycles on abelian varieties and behavior under isogeny}

\author[*]{Evangelia Gazaki*} \address[*]{\normalfont Department of Mathematics, University of Virginia,  221 Kerchof Hall, 141 Cabell Dr., Charlottesville, VA, 22904, USA. Email: \texttt{eg4va@virginia.edu}}
\begin{document}

\maketitle

\begin{abstract}
For an abelian variety  $A$ over a field $k$  the author defined in \cite{Gazaki2015} a Bloch-Beilinson type filtration $\{F^r(A)\}_{r\geq 0}$ of the Chow group of zero-cycles, $\CH_0(A)$, with successive quotients related to a Somekawa $K$-group. 
In this article we show that this filtration behaves well with respect to isogeny, and in particular if $n:A\to A$ is the multiplication by $n$ map on $A$, then its push-forward $n_\star$ is given on the quotient $F^r/F^{r+1}$ by multiplication by $n^r$. In the special case when $A=E_1\times\cdots\times E_d$ is a product of elliptic curves, we show that this filtration agrees with a filtration defined by Raskind and Spiess and with the Pontryagin filtration previously considered by Beauville and Bloch. We also obtain some results in the more general case when $A$ is isogenous to a product of elliptic curves. When $k$ is a finite extension of  $\Q_p$, using Jacobians of curves isogenous to products of elliptic curves, we give new evidence for a conjecture of Raskind and Spiess and Colliot-Th\'{e}l\`{e}ne, which predicts that the kernel of the Albanese map is the direct sum of a divisible group and a finite group.

% new infinite classes of varieties that satisfy  under some assumptions on the reduction type of $E_i$, we show that the Albanese kernel of the abelian variety $A$ is the direct sum of its maximal divisible subgroup with a torsion group. For a product $C_1\times\cdots\times C_d$ of curves 

% When $k$ is a finite extension of $\Q_p$, we show that under some assumptions on the reduction type of $E_i$ the Albanese kernel of the abelian variety $A$ is the direct sum of its maximal divisible subgroup and a torsion group of finite exponent, which is in fact finite when $A$ is an abelian surface. This gives new evidence for a conjecture of Raskind and Spiess and Colliot-Th\'{e}l\`{e}ne. 
%Using a construction due to Scholten (\cite{Scholten}) we describe explicit new infinite classes of Jacobians of hyperelliptic curves of genus $2$ that satisfy this conjecture. 
%Moreover, using the Kummer surface associated to a product $E_1\times E_2$ of two elliptic curves with fully rational $2$-torsion we describe explicit examples of hyperelliptic curves whose Jacobian satisfies  
\end{abstract}


\section{Introduction} 
Let $X$ be a smooth projective variety over a field $k$. We consider the Chow group of zero-cycles $\CH_0(X)$. The conjectures of Beilinson and Bloch (\cite{Beilinson1984, Bloch1984}) predict that this group has a finite filtration 
\[\CH_0(X)\supset F^1(X)\supset F^2(X)\supset\cdots\supset F^N(X)\supset 0\] arising from a spectral sequence in the conjectural abelian category $\mathcal{MM}$ of mixed motives. The first two pieces of the filtration are well-known; namely  
$F^1(X)$ is the kernel of the degree map, $\CH_0(X)\xrightarrow{\deg}\Z$, and $F^2(X)$ is the kernel of the Albanese map, $F^1(X)\xrightarrow{\alb_X}\Alb_X(k)$. For products of curves and abelian varieties such filtrations are classically known for the group $\CH_0(X)\otimes\Q$ by using the decomposition of the diagonal. In fact, for a product of curves $X=C_1\times\cdots\times C_d$, Raskind and Spiess  \cite{Raskind/Spiess2000}) constructed an  integral filtration, \[\CH_0(X)\supset F^1(X)\supset\cdots\supset F^d(X),\] with successive quotients isomorphic to a certain Somekawa $K$-group (\cite{Somekawa1990}). 

For an abelian variety $A$ on the other hand, a natural filtration $\{G^r(A)\}_{r\geq 0}$ of $\CH_0(A)$ arises  by using the \textit{Pontryagin product} of $\CH_0(A)$. The filtration $\{G^r(A)\otimes\Q\}_{r\geq 0}$ has been previously studied extensively by Beauville (\cite{Beauville1986}) and Bloch (\cite{Bloch1976}). Beauville proved the following two key properties. First, that $G^{d+1}(A)\otimes\Q=0$, where $d=\dim(A)$. Second, for every positive integer $n$ the push-forward $\CH_0(A)\otimes\Q\xrightarrow{n_\star} \CH_0(A)\otimes\Q$ of the multiplication by $n$ map $A\xrightarrow{n}A$ is given on the successive quotients $(G^r(A)/G^{r+1}(A))\otimes\Q\xrightarrow{n_\star} (G^r(A)/G^{r+1}(A))\otimes\Q$ by mutliplcation by $n^r$. 

%In \cite{Raskind/Spiess2000}) Raskind and Spiess constructed a finite integral filtration, \[\CH_0(X)\supset F^1(X)\supset\cdots\supset F^d(X),\] for a product of curves $X=C_1\times\cdots\times C_d$.
 In \cite{Gazaki2015} the author constructed an integral descending (a priori infinite) filtration $\{F^r(A)\}_{r\geq 0}$ of $\CH_0(A)$  for an abelian variety $A$, which agrees with the Beauville-Bloch filtration $G^r(A)$ after $\otimes\Q$. The distinctive property of this new construction is that the successive quotients $F^r/F^{r+1}$ have an interpretation as a Somekawa $K$-group. We note that it has been shown by Kahn and Yamazaki (cf.~\cite[(1.1)]{KahnYamazaki2013}) that these $K$-groups have the correct motivic interpretation as motivic cohomology groups\footnote{In fact, they are $\Ext$ groups in the category of effective motivic complexes of Voevodsky.}. Thus, to the extend to which the category $\mathcal{MM}$ has been constructed, the above constructions provide good candidates for the Bloch-Beilinson filtration. 

The first purpose of this article is to prove some key properties of the filtration $\{F^r(A)\}_{r\geq 0}$ constructed in \cite{Gazaki2015}.  
\begin{prop}\label{isogenyintro} (cf.~\autoref{isogeny1}) The filtration $\{F^r(A)\}_{r\geq 0}$ behaves well with respect to isogeny. Namely, if $\varphi:A\to B$ is an isogeny between abelian varieties, then $\varphi_\star(F^r(A))\subset F^r(B)$, for all $r\geq 0$. 
\end{prop}
The proof of \autoref{isogenyintro} yields the following important corollary. 
\begin{cor}\label{n*intro} (cf.~\autoref{n*}) For every $n\geq 1$, the push-forward $n_\star:\CH_0(A)\to\CH_0(A)$ is given on $F^r(A)/F^{r+1}(A)$ by multiplication by $n^r$. 
\end{cor}
This result shows that the key motivic property of $\{G^r(A)\otimes\Q\}_{r\geq 0}$ proved by Beauville holds integrally for the filtration $\{F^r(A)\}_{r\geq 0}$.

Next, we compare the filtration $\{F^r(A)\}_{r\geq 0}$ to the Raskind-Spiess filtration in the case of overlap, namely when $X=E_1\times\cdots\times E_d$ is a product of elliptic curves. In order to distinguish between the two constructions we will denote by $\{F^r_{\RS}(X)\}_{r\geq 0}$ the Raskind-Spiess filtration.
% In \autoref{comparesection} we prove the following results.
\begin{theo}\label{ellipticcompareintro} (cf.~\autoref{ellipticcompare}, \autoref{special fields}) Suppose $X=E_1\times\cdots\times E_d$ is a product of elliptic curves. Then the filtrations $F^r(X), G^r(X), F^r_{\RS}(X)$ all agree after $\otimes\Z[1/2]$. The result holds integrally when the base field $k$ is algebraically closed, or a finite extension of $\Q_p$, where $p$ is an odd prime and the elliptic curves $E_1,\ldots,E_d$ have good reduction. 
\end{theo}
In the more general case of an abelian variety $A$ isogenous to a product of elliptic curves we obtain the following result. 

\begin{theo}\label{maincompareintro} (cf.~\autoref{maincompare}) Let $A$ be a self-dual abelian variety admitting an isogeny $\varphi:A\to X$ to a product $X$ of elliptic curves. Suppose that $\check{\varphi}\circ\varphi=n$, where $\check{\varphi}$ is the dual isogeny. Then $\displaystyle F^r(A)\otimes\Z\left[\frac{1}{2n}\right]=G^r(A)\otimes\Z\left[\frac{1}{2n}\right]$ for all $r\geq 0$. 
\end{theo}

 We note that the proof of \autoref{ellipticcompareintro} is not immediate. It requires various properties of the Somekawa $K$-group and a careful comparison of the two constructions in \cite{Raskind/Spiess2000} and \cite{Gazaki2015}, which are a priori very different. Lastly, \autoref{maincompareintro} improves in this special case the equality $F^r(A)\otimes\Q=G^r(A)\otimes\Q$ shown in \cite[Corollary 4.4]{Gazaki2015}. It is not unreasonable to expect that the filtration $\{F^r(A)\}_{r\geq 0}$ agrees integrally with the Beauville-Bloch filtration and an integral vanishing $G^{d+1}(A)=0$ holds (cf.~\autoref{interchange}).

\subsection{Results over $p$-adic fields} We next focus on the case when  $k$ is a finite extension of the $p$-adic field $\Q_p$. In this case for a smooth projective variety $X$ over $k$ the kernel $F^2(X)$ of the Albanese map is conjectured to have the following structure.
\begin{conj}\label{mainconj} (cf.~\cite[Conjecture 3.5.4]{Raskind/Spiess2000}, \cite[1.4 (d),(e),(f)]{Colliot-Thelene1995}) The group $F^2(X)$ is the direct sum of a divisible group and a finite group. 
\end{conj}
To simplify the statements that follow we introduce some notation (cf.~\autoref{nd}). The maximal divisible subgroup of $F^2(X)$ injects in it as a direct summand. We denote by $F^2(X)_{\nd}$ the non-divisible quotient, viewed as a direct summand of $F^2(X)$. \autoref{mainconj} is then equivalent to saying that $F^2(X)_{\nd}$ is a finite group.
 This conjecture is due to Raskind and Spiess (\cite{Raskind/Spiess2000}) motivated by earlier considerations of Colliot-Th\'{e}l\`{e}ne (\cite{Colliot-Thelene1995}). 
A weaker version of this conjecture has been proved by S. Saito and K. Sato (\cite{Saito/Sato2010}), who proved that the group $F^1(X)$ is the direct sum of a finite group and a group divisible by any integer $m$ coprime to $p$. 
 Raskind and Spiess proved \autoref{mainconj} for products $X=C_1\times\cdots\times C_d$ of curves all of whose Jacobians have a mixture of good ordinary and split multiplicative reduction (\cite{Raskind/Spiess2000}). In \cite{Gazaki/Leal2022} the author in joint work with I. Leal proved the conjecture for products $X=E_1\times\cdots\times E_d$ of elliptic curves allowing one curve to have good supersingular reduction (cf.~\cite[Theorem 1.2]{Gazaki/Leal2022}). Moreover, in the same article we showed that when all elliptic curves have good ordinary reduction, the $p$-adic cycle map to \'{e}tale cohomology becomes injective after a finite base change (cf.~\cite[Theorem 1.4]{Gazaki/Leal2022}).  In \autoref{padicsection} we extend these results to the following situation.

%Another case to consider is a product of curves $X=C_1\times\cdots\times C_d$ such that each Jacobian variety $J_i$ is isogenous to a product of elliptic curves. The following proposition follows directly from \cite[Theorem 1.2, Theorem 1.4]{Gazaki/Leal2022}. 
%The following proposition follows directly from \cite[Theorem 1.4]{Gazaki/Leal2022}).  

\begin{theo}\label{mainintro2} Let $X=C_1\times\cdots\times C_d$ be a product of smooth projective curves over a $p$-adic field $k$, where $p$ is an odd prime.  Let $J_i$ be the Jacobian variety of $C_i$ for $i=1,\ldots,d$. Suppose that there exists a finite extension $K/k$ over which for each $i\in\{1,\ldots,d\}$ there exists an isogeny $\psi_i:J_i\otimes_k K\to E_{i_1}\times\cdots\times E_{i_{r_i}}$  to a product $Y_i:=E_{i_1}\times\cdots\times E_{i_{r_i}}$ of elliptic curves all of which have either good reduction or split multiplicative reduction. Suppose that for each $i\in\{1,\ldots,d\}$, $\check{\psi}_i\circ\psi_i=n_i$ is an integer coprime to $p$. 
\begin{enumerate}
\item Suppose that for at most one $i\in\{1,\ldots,d\}$ there exist indexes $i_s$ with $s\in\{1,\ldots, r_i\}$ such that the elliptic curve $E_{i_s}$ has good supersingular reduction.  
%the abelian variety $Y_i$ has coordinates with good supersingular reduction.
 Then the group $F^2(X)_{\nd}$ is torsion of finite exponent and there exists a finite extension $K'/K$ such that $F^2(X\otimes_k K')_{\nd}$ is a finite group. When $\dim(X)=2$, the group $F^2(X)_{\nd}$ is finite without having to extend the base field.
\item Under the same assumptions as in part (1), if $K=k$ is an unramified extension of $\Q_p$ and all the elliptic curves involved have good reduction, then $F^2(X)_{\nd}$ is zero. 
\item Suppose that each $Y_i$ has potentially good ordinary reduction. Then there exists a finite extension $L/k$ such that the cycle class map 
\[\CH_0(X_L)/p^n\xrightarrow{c_{p^n}} H^{2d}_{\text{\'{e}t}}(X_L,\mu_{p^n}^{\otimes d})\] to \'{e}tale cohomology is injective for every $n\geq 1$. 
\end{enumerate}
\end{theo} 
 \autoref{mainintro2} provides new infinite classes of examples that satisfy \autoref{mainconj}. We note that the above theorem and the results obtained in \cite{Raskind/Spiess2000, Gazaki/Leal2022} are the only evidence we have for the full \autoref{mainconj}. 
Additionally, the injectivity of the cycle map allows us to often fully compute the non-divisible summand of $F^2(X)$. 
 
 For an abelian variety $A$ isogenous to such a product of elliptic curves we obtain the following weaker result, using the computations from \autoref{comparesection}. 
 
\begin{prop}\label{mainabintro} (cf.~\autoref{mainab}) 
Let $A$ be a self-dual abelian variety 
 over a $p$-adic field $k$, where $p$ is an odd prime. Suppose that $A$ admits an isogeny $\varphi:A\to E_1\times\cdots\times E_d$ to a product of elliptic curves such that $\check{\varphi}\circ\varphi=n$ is an integer coprime to $p$. Assume further that each $E_i$ has good reduction and at most one $E_i$ has good supersingular reduction. Then $F^2(A)_{\nd}$ is torsion. 
 \end{prop} 
 This result generalizes \cite[Theorem 1.2]{Gazaki2019}, allowing some degree of supersingularity. To relate this to the results of the previous subsection, we note that if an integral vanishing $G^{d+1}(A)=0$ could be proved for the abelian variety $A$, this would automatically improve these results, implying that \autoref{mainconj} is true for $A$. 
 
%If we drop the assumption $X(k)\neq\emptyset$, then a weaker form of \autoref{mainintro2} (1) holds. Namely, that $F^2(X)$ is the direct sum of its maximal divisible subgroup and a torsion group of finite exponent. The same holds if we allow the elliptic curves to attain good or split multiplicative reduction after a finite base change. 
%We will denote by $F^2(A)_{\nd}$ the non-divisible summand of $F^2(A)$ (cf.~\autoref{nd}). 
%We note that for an abelian variety $A$ over $k$ with good ordinary reduction it follows by \cite[Theorem 1.2]{Gazaki2019} that the group $F^2(A)_{\nd}$ is torsion. \autoref{main1} improves this result in the special case an isogeny exists to a product of elliptic curves, showing that $F^2(A)_{\nd}$ is of bounded torsion. More importantly, this theorem allows some degree of supersingularity, which gives rise to new infinite classes of varieties that provide evidence for \autoref{mainconj}. In the next subsection we will discuss some explicit examples. 
\subsection*{Finding explicit Examples}\label{examples}  Let $X=E_1\times E_2$ be a product of two elliptic curves over a number field $k$ with fully rational $2$-torsion. We may assume that $E_1, E_2$ are given by Weierstrass equations of the form $y^2=x(x-a)(x-b)$ and $y^2=x(x-c)(x-d)$ respectively. 
Scholten (cf.~\cite[Theorem 1]{Scholten}), using the Kummer surface associated to $X$, constructed a hyperelliptic curve $C$ given by the equation
$(ad-bc)y^2=((a-b)x^2-(c-d))(ax^2-c)(bx^2-d)$, which admits dominant morphisms 
$h_i:C\to E_i$ for $i=1,2$. In the most general case the curve $C$ is smooth of genus $2$ (cf.~\cite[p.~6]{Scholten}), and hence its Jacobian $J$ is an abelian surface  admitting an isogeny $J\xrightarrow{\varphi} E_1\times E_2$. One can use this construction and variations of it to obtain infinitely many new classes of examples that satisfy \autoref{mainconj}. In \autoref{computations} we give some explicit computations.

\subsection{Notation}  For an abelian group $B$ we will denote by $B_{\dv}$ its maximal divisible subgroup, and for every integer $n\geq 1$ we will denote by $B_n$ the subgroup of $n$-torsion points of $B$. When $A$ is an abelian variety we will instead use the notation $A[n]$.  For a variety $X$ over a field $k$ we will denote by $k(X)$ the function field of $X$ and for a closed point $x\in X$, $k(x)$ will be its residue field. For a field extension $K/k$ we will denote by $X_K$ the base change to $K$ and by $\pi_{K/k}:X_K\to X$ the projection. This projection induces a flat pullback $\pi_{K/k}^\star:\CH_0(X)\to\CH_0(X_K)$, and if $K/k$ is finite, also a proper push-forward $\pi_{K/k\star}:\CH_0(X_K)\to\CH_0(X)$.  For a variety $X$ over $k$ we will denote by $X^{(0)}$ the set of closed points of $X$. For varieties $X_1,\ldots, X_r$ over $k$ the product $X_1\times\cdots X_r$ will always be assumed to be over $\Spec k$. 
%Throughout this article $k$ will be a finite extension of the $p$-adic field $\Q_p$. We will denote by $\mathcal{O}_k$ the ring of integers of $k$ and by $\F$ its residue field.

\subsection{Acknowledgement} The author's research was partially supported by the NSF grants DMS \#2001605 and DMS \#2305231. I would like to thank Jonathan Love for useful discussions. I am also truly grateful to the referee, whose substantial and detailed comments helped improve significantly the paper. 

\vspace{3pt}
\section{Comparing filtrations}\label{comparesection}
In this section we compare the various filtrations of zero-cycles for abelian varieties. 
We start by reviewing the definitions of the filtrations considered in \cite{Bloch1976, Beauville1986, Raskind/Spiess2000, Gazaki2015}. 
\subsection{Background} 
In this subsection $k$  will be an arbitrary perfect field. 
\subsection*{The Somekawa $K$-group} 
Let $A_1,\ldots, A_r$ be abelian varieties over $k$. The Somekawa $K$-group $K(k;A_1,\ldots, A_r)$ attached to $A_1,\ldots, A_r$ is defined to be the quotient 
\[K(k;A_1,\ldots, A_r):=\frac{(A_1\otimes^M\cdots\otimes^M A_r)(k)}{(\mathbf{WR})},\] where $(A_1\otimes^M\cdots\otimes^M A_r)(k)$ is the product of Mackey functors (cf.~\cite[Definition 2.7]{GazakiLove2022}) and $(\mathbf{WR})$ is a relation known as \textit{Weil reciprocity}. The generators of $K(k;A_1,\ldots, A_r)$ are denoted as symbols $\{a_1,\ldots,a_r\}_{L/k}$ where $L/k$ ranges through all finite extensions and $a_i\in A_i(L)$ for $i=1,\ldots,r$. The proof of \autoref{skew} will require the precise definition of the subgroup $(\mathbf{WR})$, which we review here. Let $C$ be a smooth projective curve over $k$. 
		For each closed point $x\in C^{(0)}$,  let $\iota_x\in C(k(x))$ be the canonical inclusion $\Spec k(x)\to C$ (induced by the quotient map $\mathcal{O}_{C,x}\to k(x)$ and the embedding $\Spec \mathcal{O}_{C,x}\to C$, where $\mathcal{O}_{C,x}$ is the local ring at $x$), and let the extension $k\to k(x)$ be determined by composing $\iota_x$ with the structure morphism $C\to\Spec k$.  
		Suppose that for each $i\in\{1,\ldots,r\}$ there are regular maps $g_i:C\to A_i$, so that $g_i\circ \iota_x\in A_i(k(x))$. Then for every $f\in k(C)^\times$ we require
			
			\begin{equation}\label{WR}
				\sum_{x\in C^{(0)}}\ord_x(f) \{g_1\circ \iota_x,\ldots, g_r\circ \iota_x\}_{k(x)/k}\in(\mathbf{WR}).
			\end{equation} 
			
\begin{notn}  Let $A$ be an abelian variety over $k$. For $r\geq 1$ we denote by $K_r(k;A)$ the Somekawa $K$-group $K(k;\underbrace{A,\ldots,A}_r)$ attached to $r$ copies of $A$. Moreover, we denote by $S_r(k;A)$ the quotient of $K_r(k;A)$ by the action of the symmetric group $\Sigma_r$ in $r$ variables. For $r=0$ we set $K_0(k;A)=S_0(k;A)=\Z$.  Note that for $r=1$ we have isomorphisms $K_1(k;A)=S_1(k;A)\simeq A(k)$ (cf.~\cite[Corollary 3.7]{Gazaki2015}). 
\end{notn}
\begin{notn}\label{closedpoint} Let $Y$ be a smooth projective variety over $k$ and $y\in Y(k)$. Then $y$ gives rise to a closed point $\tilde{y}$ of $Y$. We will abuse notation and denote this closed point by $y$ and by $[y]$ the induced zero-cycle in $\CH_0(Y)$. More generally, let $L/k$ be a finite extension and let $y\in Y(L)$. Then $y$ induces a zero-cycle $[y]_L\in\CH_0(X_L)$ of degree one and a zero-cycle $\pi_{L/k\star}([y]_L)\in\CH_0(X)$ of degree $[L:k]$. 
\end{notn}

\subsection*{The Raskind-Spiess filtration}

 Let $X=E_1\times\cdots\times E_d$ be a product of elliptic curves over $k$.  Raskind and Spiess (\cite{Raskind/Spiess2000}) showed an isomorphism 
\[f:\bigoplus_{\nu=0}^d\bigoplus_{1\leq i_1<\cdots<i_\nu\leq d}K(k;E_{i_1},\ldots, E_{i_\nu})\xrightarrow{\simeq}\CH_0(X).\] We will make this isomorphism more explicit. 
For $i\in\{1,\ldots,d\}$ we will denote by $0_i\in E_i(k)$ the zero element of $E_i$, and for a finite extension $L/k$, $0_{iL}$ will be the zero element of $(E_{i})_L$, which is nothing but $\res_{L/k}(0_i)$. Let $\pi_{iL}:X_L\to (E_{i})_L$ be the projection for $i=1,\ldots,d$. These induce pullback maps $\pi_{iL}^\star:\Pic((E_{i})_L)\to\Pic(X_L)$. Each rational point $x_i\in E_{i}(L)$ induces a divisor $[x_i]_L\in\Pic((E_{i})_L)$ of degree one (cf.~\autoref{closedpoint}). 
 Let $1\leq i_1<\cdots<i_\nu\leq d$. If follows by the proofs of \cite[Theorem 2.2]{Raskind/Spiess2000} and \cite[Corollary 2.4.1]{Raskind/Spiess2000} that the homomorphism $f$ when  restricted to $K(k;E_{i_1},\ldots, E_{i_\nu})$ has the following explicit description: 
\begin{eqnarray*}
&&f: K(k;E_{i_1},\ldots, E_{i_\nu})\rightarrow \CH_0(X)\\
&&\;\;\;\;\;\;\{x_{i_1},\ldots,x_{i_\nu}\}_{L/k}\mapsto \pi_{L/k\star}( \pi_{1L}^\star(D_1)\cdot \pi_{1L}^\star(D_2)\cdot\cdots \pi_{1L}^\star(D_d)), 
\end{eqnarray*} where $\cdot$ is the intersection product and for $i=1,\ldots, d$ the divisor $D_i\in \Pic((E_{i})_L)$ is defined to be $D_{i_t}=[x_{i_t}]_L-[0_{i_t}]_L$ if $t\in\{i_1,\ldots,i_\nu\}$ and $D_i=[0_{i}]_L$ otherwise. 
Let $r\in\{1,\ldots,d\}$. Define $\displaystyle F^r_{\RS}(X):=f\left(\bigoplus_{\nu=r}^d\bigoplus_{1\leq i_1<i_2<\cdots<i_\nu\leq d}K(k;E_{i_1},\ldots, E_{i_\nu})\right)$, so that $f$ induces an isomorphism 
\begin{equation}\label{RSiso}F^r_{\RS}(X)/F^{r+1}_{\RS}(X)\simeq \bigoplus_{1\leq i_1<\cdots<i_r\leq d}K(k;E_{i_1},\ldots, E_{i_r}),
\end{equation} and $F^{d+1}_{\RS}(X)=0$. We note that $F^1_{\RS}(X)=\ker(\deg)$ and $F^2_{\RS}(X)=\ker(\alb_X)$.

\subsection*{The Gazaki filtration}
Next, let $A$ be an abelian variety over $k$. It follows by \cite[Proposition 1.1]{Gazaki2015} that for every $r\geq 0$ there is a well-defined homomorphism
\[\Phi_r:\CH_0(A)\to S_r(k;A),\;[a]\mapsto\{a,a,\ldots,a\}_{k(a)/k}.\] We define $F^0(A)=\CH_0(A)$ and for $r\geq 1$,  
$\displaystyle F^r(A):=\bigcap_{j=0}^{r-1}\ker(\Phi_r).$ Thus, we have an embedding 
\[\Phi_r: F^r(A)/F^{r+1}(A)\hookrightarrow S_r(k;A).\] Moreover, it follows by \cite[Proposition 1.2]{Gazaki2015} that for all $r\geq 0$ there is a homomorphism
\[\Psi_r:S_r(k;A) \rightarrow F^r(A)/F^{r+1}(A)\] with the property $\Phi_r\circ\Psi_r=r!$. Thus, the map $\Phi_r$ becomes an isomorphism after $\otimes\Z[1/r!]$. In what follows we won't need the precise definition of $\Psi_r$, only its properties. 

\subsection*{The Pontryagin filtration}
The filtration $\{F^r(A)\}_{r\geq 0}$ is tightly related to another well-studied filtration $\{G^r(A)\}_{r\geq 0}$ of $\CH_0(A)$ previously considered by Beauville and Bloch (\cite{Beauville1986, Bloch1976}). This filtration is obtained as follows. 
\begin{eqnarray*}
&&G^0(A)=\CH_0(A)\\
&&G^{1}(A)=\langle\pi_{L/k\star}([a]_{L}-[0]_{L}):a\in A(L)\rangle,\\
&&G^{2}(A)=\langle\pi_{L/k\star}([a+b]_{L}-[a]_{L}-[b]_{L}+[0]_{L}):a,\;b\in A(L)\rangle,\\
&&...\\
&&G^{r}(A)=\left\langle\sum_{j=0}^{r}(-1)^{r-j}\pi_{L/k\star}\left(\sum_{1\leq\nu_{1}<\dots<\nu_{j}\leq r}[a_{\nu_{1}}+\dots+a_{\nu_{j}}]_{L}\right):a_{1},\dots,a_{r}\in A(L)\right\rangle. 
\end{eqnarray*}  It follows by \cite[Proposition 3.3]{Gazaki2015} that for every $r\geq 0$ we have an inclusion $G^r(A)\subseteq F^r(A)$, which becomes an isomorphism after $\otimes\Q$ (cf.~\cite[Corollary 4.4]{Gazaki2015}). 
The reason for the name Pontryagin filtration is because the filtration $\{G^r(A)\}_{r\geq 1}$ is induced by the Pontryagin product $\odot$ on $\CH_0(A)$, a structure which makes $(\CH_0(A), +,\odot)$ a commutative unital ring. 
We briefly recall how $\odot$ is defined. Let $m:A\times A\to A$ be the multiplication map, and $\pr_i:A\times A\to A$ be the projection to the $i$-th factor for $i=1,2$. Then for $\alpha,\beta\in\CH_0(A)$ we have,
\[\alpha\odot\beta:=m_\star(\pr_1^\star(\alpha)\cdot\pr_2^\star(\beta)).\] 
Under this product the degree map $\deg:\CH_0(A)\to\Z$ becomes a ring homomorphism. 
What is important to us in this paper is that if $a,b\in A(k)$, then for the induced zero-cycles $\alpha=[a], \beta=[b]\in\CH_0(A)$, the product is given by  
\[[a]\odot[b]=[a+b]-[a]-[b]+[0].\] Using this we can rewrite the generators of $G^r(A)$ as follows,
\begin{equation}\label{product1}
G^r(A)=\left\langle\sum_{j=0}^{r}(-1)^{r-j}\pi_{L/k\star}\left(
([a_1]_L-[0]_L)\odot\cdots([a_r]_L-[0]_L)\right):a_{1},\dots,a_{r}\in A(L)\right\rangle.
\end{equation}

% Namely, the group structure on the abelian variety $A$ induces a product $\odot$ on $\CH_0(A)$ making $(\CH_0(A), +,\odot)$ a commutative unital ring. 

%The abelian group structure on $A(L)$, where $L/k$ is any finite extension, gives $\CH_0(A)$ the structure of a commutative unital ring equipped with a ring homomorphism $\deg:\CH_0(A)\to\Z$.
% Let $a,b\in A(L)$. Consider the induced zero-cycles $[a]_L,[b]_L\in\CH_0(A_L)$ (cf.~\autoref{closedpoint}). The \textit{Pontryagin product} $\odot$ is defined by 
%%$[a]\star[b]:=[a+b]$. More generally, if $a,b\in A(L)$ for some finite extension $L/k$ we define 
%\[\pi_{L/k\star}([a]_L)\odot\pi_{L/k\star}([b]_L):=\pi_{L/k\star}([a+b]_L)\in\CH_0(A).\] We have $G^{0}(A)=\CH_{0}(A)$. The degree map $\CH_0(A)\xrightarrow{\deg}\Z$ can be identified with the augmentation map of this group ring and $G^1(A)=F^1(A)=\ker(\deg)$ is the kernel of the augmentation ideal. Then  $\{G^r(A)\}_{r\geq 1}$ is given by the powers of this augmentation ideal.
%We note that for $r\geq 1$ the group $G^r(A)$ has the following explicit generators, 

\begin{rem} In the case $k=\overline{k}$ the group $G^r(A)$ is the $r$-th power of $G^1(A)$, with the latter being the augmentation ideal with respect to the degree map. It was pointed to us by the referee that this might not be true when $k\neq\overline{k}$. 
\end{rem}

The following easy lemma will be used in the proof of \autoref{maincompareintro}. 
\begin{lem}\label{Grisogeny} Let $A,B$ be isogenous abelian varieties over $k$ via an isogeny $\varphi:A\to B$. Then $\varphi_\star(G^r(A))\subset G^r(B)$ for all $r\geq 0$. 
\end{lem} 
\begin{proof}
Let $L/k$ be a finite extension. We denote by $\varphi_L:A_L\to B_L$ the isogeny obtained by base change of $\varphi$. Moreover, we denote by $\pi_{L/k,A}:A_L\to A$,  $\pi_{L/k,B}:B_L\to B$ the projections. We have an equality $\varphi\circ\pi_{L/k,A}=\pi_{L/k,B}\circ\varphi_L$. For $z\in \CH_0(A_L)$ we have,
\[\varphi_\star(\pi_{L/k,B\star}(z))=(\varphi\circ\pi_{L/k,B})_\star(z)=\pi_{L/k,B\star}(\varphi_{L\star}(z)).\] Let $z\in G^r(A)$. We will show $\varphi_\star(z)\in G^r(B)$. Using the explicit generators of $G^r(A)$ and the above relation, it is enough to prove the claim for zero-cycles of the form $\displaystyle z=\sum_{j=0}^{r}(-1)^{r-j}\left(\sum_{1\leq\nu_{1}<\dots<\nu_{j}\leq r}[a_{\nu_{1}}+\dots+a_{\nu_{j}}]\right)$ with $a_{\nu_i}\in A(k)$ for all $\nu_i$. Since all closed points involved in this expression are defined over the base field, it follows that
\begin{eqnarray*}\varphi_\star(z)=&&\sum_{j=0}^{r}(-1)^{r-j}\left(\sum_{1\leq\nu_{1}<\dots<\nu_{j}\leq r}[\varphi(a_{\nu_{1}}+\dots+a_{\nu_{j}})]\right)\\=&&\sum_{j=0}^{r}(-1)^{r-j}\left(\sum_{1\leq\nu_{1}<\dots<\nu_{j}\leq r}[\varphi(a_{\nu_{1}})+\dots+\varphi(a_{\nu_{j}})]\right).
\end{eqnarray*} The second equality follows because $\varphi$ is a homomorphism. The last expression yields that $\varphi_\star(z)\in G^r(B)$. 

\end{proof}

\begin{lem}\label{compare1} Let $X=E_1\times\cdots\times E_d$ be a product of elliptic curves over $k$. Then we have an inclusion $F^r_{\RS}(X)\subseteq G^r(X)\subseteq F^r(X)$ for $r\geq 0$. 
\end{lem}
\begin{proof}  To prove $F^r_{\RS}(X)\subseteq G^r(X)$ for $r\geq 1$ it is enough to show that the image of the homomorphism $f:K(k;E_{i_1},\ldots, E_{i_\nu})\rightarrow \CH_0(X)$ is contained in $G^r(X)$ for every $\nu\in\{r,\ldots,d\}$ and every $1\leq i_1<\ldots<i_\nu\leq d$. Let $\{x_{i_1},\ldots,x_{i_\nu}\}_{L/k}$ be a generator of $K(k;E_{i_1},\ldots, E_{i_\nu})$. The explicit description of the homomorphism $f$ described earlier gives
\[f(\{x_{i_1},\ldots,x_{i_\nu}\}_{L/k})=\pi_{L/k\star}\left(([a_1]_L-[0]_L)\odot\cdots\odot([a_\nu]_L-[0]_L)\right),\] where the zero-cycle $[a_{t}]_L$ is the one induced by the $L$-rational point $a_t=(0,\ldots, x_{i_t},\ldots, 0)$ for $t\in\{1,\ldots,\nu\}$. Then by definition of the filtration $\{G^i(A)\}_{i\geq 0}$ we have \[f(\{x_{i_1},\ldots,x_{i_\nu}\}_{L/k})\in G^\nu(A)\subset G^{r}(A)\] as required. 
%The explicit description of the homomorphism $f$ described earlier gives, 
%\[f(K(k;E_{i_1},\ldots, E_{i_\nu})=\langle[(0,\ldots,x_{i_1},0,\ldots,x_{i_\nu},0,\ldots, 0]_L\rangle\]

\end{proof}
 
% We refer to \cite[Proposition 3.3]{Gazaki2015} for a general definition of $G^r$. In the case of an algebraically closed field $k=\overline{k}$ the first pieces of the filtration are defined as follows:
%\begin{eqnarray*}
%&&G^0\CH_0(A)=\CH_0(A)\\
%&&G^1\CH_0(A)=\langle[a]-[0]:a\in A\rangle=\ker(\deg)=F^1(A)\\
%&&G^2\CH_0(A)=\langle[a+b]-[a]-[b]+[0]:a,b\in A\rangle=\ker(\alb_A)=F^2(A)\\
%&&G^3\CH_0(A)=\langle[a+b+c]-[a+b]-[b+c]+[a]+[b]+[c]-[0]:a,b,c\in A\rangle.
%\end{eqnarray*} It follows by \cite[Proposition 3.3]{Gazaki2015} that there is an injection $G^r(A)\hookrightarrow F^r(A)$, for all $r\geq 0$.

\subsection{Product of Elliptic Curves}\label{ellipticsection} In this subsection we prove  \autoref{ellipticcompareintro}. 
%\subsection{Proof of \autoref{maincompareintro}}
 We start with some preparatory results. 

\begin{lem}\label{Kr} Let $A_1,\ldots, A_d$ be abelian varieties over $k$. For  $r\geq 1$ there is an isomorphism 
\begin{equation}\label{iso1}
h_r:K_r(k;A_1\times\cdots\times A_d)\xrightarrow{\simeq}\bigoplus_{(i_1,\ldots,i_r)\in\{1,\ldots, d\}^r}K(k;A_{i_1},\ldots, A_{i_r}).\end{equation}
\end{lem}
\begin{proof}
For $i\in\{1,\ldots, d\}$ let $\pr_i: A_1\times\cdots\times A_d\to A_i$ be the projection on the $i$ factor and $\varepsilon_i: A_i\to A_1\times\cdots\times A_d$ be the embedding sending $a_i\in A_i$ to $(0,\ldots, a_i,\ldots, 0)$. Since Somekawa $K$-groups satisfy covariant functoriality with respect to homorphisms of abelian varieties (cf.~\cite[p.~108]{Somekawa1990}) the maps $\pr_i$ and $\varepsilon_i$ induce homomorphisms $\varepsilon_{i_1}\otimes\cdots\otimes\varepsilon_{i_r}:K(k;A_{i_1},\ldots, A_{i_r})\to K_r(k;A_1\times\cdots\times A_d)$, and $\pr_{i_1}\otimes\cdots\otimes\pr_{i_r}:K_r(k;A_1\times\cdots\times A_d)\to K(k;A_{i_1},\ldots, A_{i_r})$ with $(\pr_{i_1}\otimes\cdots\otimes\pr_{i_r})\circ (\varepsilon_{i_1}\otimes\cdots\otimes\varepsilon_{i_r})=\id$. This induces homomorphisms
\[\sum_{(i_1,\ldots,i_r)\in\{1,\ldots, d\}^r}\varepsilon_{i_1}\otimes\cdots\otimes\varepsilon_{i_r}:\bigoplus_{(i_1,\ldots,i_r)\in\{1,\ldots, d\}^r}K(k;A_{i_1},\ldots, A_{i_r})\to K_r(k;A_1\times\cdots\times A_d),\] 
\[\bigoplus_{(i_1,\ldots,i_r)\in\{1,\ldots, d\}^r}(\pr_{i_1}\otimes\cdots\otimes\pr_{i_r}): K_r(k;A_1\times\cdots\times A_d)\to \bigoplus_{(i_1,\ldots,i_r)\in\{1,\ldots, d\}^r}K(k;A_{i_1},\ldots, A_{i_r})\] with $\displaystyle\left(\bigoplus_{(i_1,\ldots,i_r)\in\{1,\ldots, d\}^r}(\pr_{i_1}\otimes\cdots\otimes\pr_{i_r})\right)\circ \left(\sum_{(i_1,\ldots,i_r)\in\{1,\ldots, d\}^r}\varepsilon_{i_1}\otimes\cdots\otimes\varepsilon_{i_r}\right)=\id$. We can easily see that the map $\displaystyle\sum_{(i_1,\ldots,i_r)\in\{1,\ldots, d\}^r}\varepsilon_{i_1}\otimes\cdots\otimes\varepsilon_{i_r}$ is surjective, from where the isomorphism follows. 
% It is a straightforward computation to show that this map is an isomorphism. 

\end{proof}

We next want to incorporate the action of the symmetric group $\Sigma_r$ in $r$-variables in the isomorphism \eqref{iso1}. Before describing the general case, we give an example that illustrates the situation. 
\begin{exmp}
When $r=d=2$ we have an isomorphism
\[K_2(k;A_1\times A_2)\simeq K_2(k;A_1)\oplus K(k;A_1,A_2)\oplus K(k;A_2,A_1)\oplus K_2(k;A_2).\] Let $\Sigma_2=\langle\sigma\rangle$ be the symmetric group in $2$ variables. We see that $\sigma$ acts on $K_2(k;A_1\times A_2)$ by 
\[\sigma\{(a_1,a_2),(a_1',a_2')\}_{L/k}=\{(a_1',a_2'),(a_1,a_2)\}_{L/k}.\] Transferring the action on the right hand side of the  isomorphism we observe the following:
\begin{itemize}
\item For $i=1,2$ we have $\sigma\cdot K_2(k;A_i)\subset K_2(k;A_i)$ and the action is given by $\sigma\{a_i, a_i'\}_{L/k}=\{a_i',a_i\}_{L/k}$. 
\item $\sigma$ acts on $K(k;A_1,A_2)\oplus K(k;A_2,A_1)$ by exchanging the two factors, $\sigma\{a_1,a_2\}_{L/k}=\{a_2, a_1\}_{L/k}$. 
\end{itemize} Thus, the group $S_r(k;A_1\times A_2)$ (which is the quotient of $K_r(k;A_1\times A_2)$ by the $\Sigma_r$-action) becomes isomorphic via $h_r$ to the group $S_2(k;A_1)\oplus K(k;A_1, A_2)\oplus S_2(k;A_2)$. 

 %of the isomorphism we see that 
%$\sigma$ acts on $K_2(k;A_i)$ by $\sigma\{a_i, a_i'\}_{L/k}=\{a_i',a_i\}_{L/k}$ and on $K($
\end{exmp}

More generally, let $\underline{i}=(i_1,\ldots, i_r)\in\{1,\ldots,d\}^r$. Let \[H_{\underline{i}}=\{\sigma\in\Sigma_r|(i_1,\ldots, i_r)=(i_{\sigma(1)},\ldots, i_{\sigma(r)})\}\] be the stabilizer of $\underline{i}$ in $\Sigma_r$. Then the isomorphism \eqref{iso1} can be rewritten as 
\begin{equation}\label{iso2} K_r(k;A_1\times\cdots\times A_d)\simeq\bigoplus_{i_1\leq i_2\leq\cdots\leq i_r\leq d}\Ind_{H_{\underline{i}}}^{\Sigma_r}K(k;A_{i_1},\ldots, A_{i_r}).
\end{equation}
\begin{defn} Let $(i_1,\ldots, i_r)\in\{1,\ldots,d\}^r$. The symmetric $K$-group $S_r(k;A_{i_1},\ldots, A_{i_r})$ attached to $A_{i_1},\ldots, A_{i_r}$ is defined to be the $\Sigma_r$-coinvariants of $\Ind_{H_{\underline{i}}}^{\Sigma_r}K(k;A_{i_1},\ldots, A_{i_r})$. 
\end{defn} 
Thus, we have
\[S_r(k;A_{i_1},\ldots, A_{i_r})=\left(\Ind_{H_{\underline{i}}}^{\Sigma_r}K(k;A_{i_1},\ldots, A_{i_r})\right)_{\Sigma_r}=K(k;A_{i_1},\ldots, A_{i_r})_{H_{\underline{i}}}.\] Notice that when $i_1<i_2<\cdots<i_r$, $H_{\underline{i}}$ is trivial on $K(k;A_{i_1},\ldots, A_{i_r})$, and hence in this case $S_r(k;A_{i_1},\ldots, A_{i_r})= K(k;A_{i_1},\ldots, A_{i_r})$. 

The following Corollary summarizes all the above. 

%\autoref{Kr} allows us to view each $K$-group $K(k;A_{i_1},\ldots, A_{i_r})$ for $(i_1,\ldots,i_r)\in\{1,\ldots, d\}$ as a subgroup of the $K$-group $K_r(k;A_1\times\cdots\times A_d)$. Consider the action of the symmetric group $\Sigma_r$ in $r$ variables on $K_r(k;A_1\times\cdots\times A_d)$. It is clear that each $K(k;A_{i_1},\ldots, A_{i_r})$ is $\Sigma_r$-invariant. We will denote by $S_r(k;A_{i_1},\ldots, A_{i_r})$ the quotient of $K(k;A_{i_1},\ldots, A_{i_r})$ by the action of the symmetric group. Notice that when the sequence $1\leq i_1<\ldots< i_r\leq d$ is strictly increasing, the action of $\Sigma_r$ is trivial, and hence in this case $S_r(k;A_{i_1},\ldots, A_{i_r})\simeq K(k;A_{i_1},\ldots, A_{i_r})$. The following corollary follows immediately from \autoref{Kr}. 

\begin{cor}\label{Sr} Let $A_1,\ldots, A_d$ be abelian varieties over $k$.  There is an isomorphism 
\[g_r:S_r(k;A_1\times\cdots\times A_d)\xrightarrow{\simeq}\bigoplus_{1\leq i_1\leq\cdots\leq i_r\leq d}S_r(k;A_{i_1},\ldots, A_{i_r}).\] In particular, the group $\bigoplus_{1\leq i_1<\cdots<i_r\leq d}K(k;A_{i_1},\ldots, A_{i_r})$ embeds as a direct summand on $S_r(k;A_1\times\cdots\times A_d)$. 
\end{cor} 

\begin{prop}\label{skew} Let $A_1,\ldots, A_r$ be abelian varieties over $k$ for some $r\geq 2$. Suppose that $A_1=A_2=E$ is an elliptic curve. The following relation holds in the Somekawa $K$-group $K(k;A_1,\ldots, A_r)$.  
\[\{a_1,a_2,\ldots,a_r\}_{L/k}=-\{a_2,a_1,\ldots,a_r\}_{L/k},\] where $L/k$ ranges over all finite extensions of $k$ and $a_i\in A_i(L)$ for $i=1,\ldots,r$. In particular, the symbol $\{a,b\}$ in $K_2(E)$ is skew-symmetric. 
%$E_1,\cdots,E_d$ be elliptic curves over $k$. Let $r\geq 1$ and  $1\leq i_1\leq\cdots\leq i_r\leq d$ be a tuple. If at least two of the indexes coincide, then the group $S_r(k;E_{i_1},\ldots, E_{i_r})$ is $2$-torsion. 
\end{prop}
\begin{proof} Note that by definition we have $\{a_1,a_2,\ldots,a_r\}_{L/k}=N_{L/k}(\{a_1,a_2,\ldots,a_r\}_{L/L})$, where $N_{L/k}$ is the norm map on Somekawa $K$-groups (cf.~\cite[p.~108]{Somekawa1990}). 
Thus, it is enough to prove the proposition for symbols $\{a_1,a_2,\ldots,a_r\}_{k/k}$ defined over $k$.   

Let $a_1, a_2\in E(k)$ and $a_i\in A_i(k)$ for $i\geq 3$. The statement is clear if either $a_1=0$ or $a_2=0$, so we may assume $a_1, a_2\neq 0$. Consider the homomorphism $\rho:E(k)\xrightarrow{\simeq}\Pic^0(E)$, $a\mapsto[a]-[0]$. It follows that $\rho(a_1+a_2)=\rho(a_1)+\rho(a_2)$. That is,  $[a_1]-[0]+[a_2]-[0]=[a_1+a_2]-[0]\in\Pic^0(E)$, and hence $[a_1]+[a_2]-[a_1+a_2]-[0]$ is the divisor of a function $f\in k(E)^\times$. We apply Weil Reciprocity $(\textbf{WR})$ of the Somekawa $K$-group for the curve $E=A_1=A_2$, the regular functions $g_i:E\to A_i$ given by $g_1=g_2=1_E$, $g_i=a_i$ constant for $i\geq 3$, and  the function $f\in k(E)^\times$ as above. The function $f$ has two simple zeros at $a_1, a_2$ and two simple poles at $0$ and $a_1+a_2$ (unless $a_1+a_2=0$, in which case the pole is double). Thus, Weil Reciprocity gives,
\begin{eqnarray*}
&&\{a_1, a_1,a_3,\ldots, a_r\}_{k/k}+\{a_2, a_2,a_3,\ldots, a_r\}_{k/k}-\{a_1+a_2, a_1+a_2,a_3,\ldots, a_r\}_{k/k}=0.
\end{eqnarray*}
Expanding the third symbol using bilinearity yields the desired relation $\{a_1,a_2,\ldots,a_r\}_{L/k}=-\{a_2,a_1,\ldots,a_r\}_{k/k}$.

\end{proof}

%\begin{rem} \autoref{skew} implies in particular that 
%\end{rem}

\begin{cor}\label{2torsion} Let $E_1,\ldots, E_d$ be elliptic curves over $k$. Let $r\geq 1$ and $1\leq i_1\leq\cdots\leq i_r\leq d$. If $i_l=i_m$ for some $l,m\in\{1,\ldots,r\}$, $l\neq m$, then the $K$-group $S_r(E_{i_1},\ldots,E_{i_r})$ is $2$-torsion. 
\end{cor}
\begin{proof}
Let $\sigma\in\Sigma_r$ be a permutation. Then there is a canonical isomorphism
\[K(k;E_{i_1},\ldots,E_{i_r})\simeq K(E_{\sigma(i_1)},\ldots,E_{\sigma(i_r)}).\] Thus we may assume $E_{i_1}=E_{i_2}$. Let $\{a_1,\ldots,a_r\}_{L/k}\in S_r(E_{i_1},\ldots,E_{i_r})$ be a generator. By definition of $S_r(E_{i_1},\ldots,E_{i_r})$ we have an equality $\{a_1,\ldots,a_r\}_{L/k}=\{a_2,a_1,\ldots,a_r\}_{L/k}$. It follows by \autoref{skew} that we also have $\{a_1,\ldots,a_r\}_{L/k}=-\{a_2,a_1,\ldots,a_r\}_{L/k}$. The claim thus follows.

\end{proof}

We are now ready to prove \autoref{ellipticcompareintro}, which we restate here. 

\begin{theo}\label{ellipticcompare}  Let $X=E_1\times\cdots \times E_d$ be a product of elliptic curves over a perfect field $k$. 
%Let $A$ be an abelian variety of dimension $d\geq 2$ over a perfect field $k$. Let $\varphi:A\to X$ be an isogeny of degree $n$ to a product  $X=E_1\times\cdots \times E_d$ of elliptic curves. The following are true. 
 The inclusion $\iota_r:F^r_{\RS}(X)\hookrightarrow F^r(X)$ of \autoref{compare1} becomes an isomorphism after $\otimes\Z[1/2]$ for all $r\geq 0$. 
%\item[(ii)] The push-forward map $\varphi_\star:\CH_0(A)\to\CH_0(X)$ induces for  $r\geq 0$ inclusions \[\varphi_\star(F^r(A)\otimes\Z[1/2])\subseteq F^r_{\RS}(X)\otimes\Z[1/2].\]   
%\item[(iii)] The group $F^{d+1}(A)$ is $n$-torsion, and hence vanishes when $k=\overline{k}$. 
 Moreover, if the Somekawa $K$-group $K(k;E_{i_1},\ldots, E_{i_r})$ is $2$-divisible for every $r\geq 2$ and every $1\leq i_1\leq\cdots\leq i_r\leq d$, then the isomorphism holds integrally. 
% \blue In the special case when $A$ is an abelian surface the theorem holds without any assumption on the base field $k$.  
%\end{enumerate}
\end{theo}

\begin{proof}
 The inclusion $\iota_r:F^r_{\RS}(X)\hookrightarrow F^r(X)$ is an equality for $r=0,1,2$. It is enough to show that $\iota_r$ induces an isomorphism $\iota_r\otimes 1:(F^r_{\RS}(X)/F^{r+1}_{\RS}(X))\otimes\Z[1/2]\xrightarrow{\simeq} (F^r(X)/F^{r+1}(X))\otimes\Z[1/2]$ for $r\geq 0$. For, the claim will follow by induction and the five lemma applied to the following commutative diagram, 
% In the case $X=E_1\times\cdots\times E_d$, it is a straightforward computation that $F^r_{\RS}(X)\subset G^r(X)$ for all $r\geq 0$. In fact, equality is true. We conclude that there is a natural inclusion $\iota_r: F^r_{\RS}(X)\hookrightarrow F^r(X)$, for $r\geq 0$, which is an equality for $i=1,2$ by properties of both filtrations. We will show that $\iota_r$ becomes an isomorphism after $\otimes\Z[1/r!]$ for all $r\geq 0$. It is enough to show that $\iota_r$ induces an isomorphism $(F^r_{\RS}(X)/F^{r+1}_{\RS}(X))\otimes\Z[1/r!]\simeq (F^r(X)/F^{r+1}(X))\otimes\Z[1/r!]$. For, the claim will follow by induction using the five lemma applied to the commutative diagram
\[\xymatrix{ 0 \ar[r] & F^{r+1}_{\RS}(X)\otimes\Z[1/2]\ar[r]\ar[d]^{\iota_{r+1}\otimes 1} & F^r_{\RS}(X)\otimes\Z[1/2]\ar[r]\ar[d]^{\iota_r\otimes 1} & F^r_{\RS}(X)/F^{r+1}_{\RS}(X)\otimes\Z[1/2]\ar[r]\ar[d]^{\iota_r\otimes 1} & 0\\
0 \ar[r] & F^{r+1}(X)\otimes\Z[1/2]\ar[r] & F^r(X)\otimes\Z[1/2]\ar[r] & F^r(X)/F^{r+1}(X)\otimes\Z[1/2]\ar[r] & 0.
}
\]

% The claim is true for $i=1,2$ by properties of both filtrations.
% \autoref{Sr} gives an isomorphism $\displaystyle S_r(k;E_1\times\cdots\times E_d)\xrightarrow{\simeq}\bigoplus_{1\leq i_1\leq\cdots\leq i_r\leq d}S_r(k;E_{i_1},\ldots, E_{i_r})$. We decompose as follows, 
Let $r\geq 2$. \autoref{Sr} gives an isomorphism
\[\displaystyle S_r(k;E_1\times\cdots\times E_d)\xrightarrow{\simeq}\bigoplus_{1\leq i_1\leq\cdots\leq i_r\leq d}S_r(k;E_{i_1},\ldots, E_{i_r})\simeq\bigoplus_{1\leq i_1<\cdots< i_r\leq d}K(k;E_{i_1},\ldots, E_{i_r})\bigoplus H,\] where $H$ is the subgroup of $S_r(k;E_1\times\cdots\times E_d)$ consisting of all direct summands of the form $S_r(k;E_{i_1},\ldots, E_{i_r})$ where at least two of the indexes are the same.
\autoref{2torsion} yields that the group $H$ is $2$-torsion, and hence $H\otimes\Z[1/2]$ vanishes. We consider the composition
\[F^r_{\RS}(X)/F^{r+1}_{\RS}\otimes\Z[1/2]\xrightarrow{\iota_r} F^r(X)/F^{r+1}(X)\otimes\Z[1/2]\xrightarrow{\Phi_r\otimes 1}S_r(k;E_1\times\cdots\times E_d)\otimes\Z[1/2].\] We recall that the homomorphism $\Phi_r$ is injective. Moreover, we have an  equality \[\displaystyle S_r(k;E_1\times\cdots\times E_d)\otimes\Z[1/2]=\bigoplus_{1\leq i_1<\cdots< i_r\leq d}K(k;E_{i_1},\ldots, E_{i_r})\otimes\Z[1/2].\] A straightforward computation shows that $\Phi_r\circ\iota_r\otimes 1$ coincides with the isomorphism \[f:F^r_{\RS}(X)/F^{r+1}_{\RS}\otimes\Z[1/2]\xrightarrow{\simeq}\bigoplus_{1\leq i_1<\cdots< i_r\leq d}K(k;E_{i_1},\ldots, E_{i_r})\otimes\Z[1/2]\] of Raskind-Spiess. Since $\Phi_r\circ\iota_r\otimes 1$ is an isomorphism and $\Phi_r$ is injective, we conclude that $\Phi_r\otimes 1$ is surjective, and hence an isomorphism. Thus, so is $\iota_r\otimes 1$. 

%We next prove (ii). 

To finish the proof of the theorem suppose that for all $r\geq 2$ and every $1\leq i_1\leq\cdots\leq i_r\leq d$ the Somekawa $K$-group $K(k;E_{i_1},\ldots, E_{i_r})$ is $2$-divisible. This implies that for all $1\leq i_1\leq\cdots\leq i_r\leq d$ with at least two of the indexes being the same the group $S_r(k;E_{i_1},\ldots, E_{i_r})$ is $2$-divisible and $2$-torsion, and hence it vanishes. Thus, the isomorphism 
$\displaystyle S_r(k;X)\simeq\bigoplus_{1\leq i_1<\cdots< i_r\leq d}K(k;E_{i_1},\ldots, E_{i_r})$ holds integrally. 

\end{proof} 

\begin{cor}\label{special fields} Consider the set-up of \autoref{ellipticcompare}. Suppose that $k$ is algebraically closed, or $k$ is a finite extension of $\Q_p$ with $p$ an odd prime and the elliptic curves $E_1,\ldots, E_d$ have good reduction. Then 
 \autoref{ellipticcompare} holds integrally. 
\end{cor}
\begin{proof}
Suppose first that $k$ is algebraically closed. Then each group $K(k;E_{i_1},\ldots, E_{i_r})$ is divisible, since it is a quotient of $E_{i_1}(k)\otimes\cdots\otimes E_{i_r}(k)$, which is divisible. Next, suppose that $k$ is a finite extension of $\Q_p$ with $p$ odd and $E_1,\ldots, E_d$ have good reduction. It follows by \cite[Corollary 3.5.1]{Raskind/Spiess2000} that the group $K(k;E_{i_1},\ldots, E_{i_r})$ is $m$-divisible for every integer $m$ coprime to $p$. In particular, it is $2$-divisible. In both cases (a) follows by \autoref{ellipticcompare}.  

%We next prove (b). It follows by \autoref{maincompare} (ii) that $\varphi_\star(F^r(A))\subseteq F^r_{\RS}(X)$ for all $r\geq 0$. In particular, $\varphi_\star(F^{d+1}(A))\subseteq F^{d+1}_{\RS}(X)=0$. Thus, $F^{d+1}(A)\subset\ker(\varphi_\star)$. Let $z\in F^{d+1}(A)$. Consider the dual isogeny $\check{\varphi}: E_1\times\cdots\times E_d\to A$ so that we have an equality $\check{\varphi}\varphi=n$. Then $nz=\check{\varphi}_\star\varphi_\star(z)=0$, and hence $F^{d+1}(A)$ is $n$-torsion. 

%Lastly, suppose that $\dim(A)=2$. It follows by \cite[Theorem 8.1]{CT93} that the $m$-torsion subgroup $\CH_0(A)_m$ of $\CH_0(A)$ is finite for every $m\geq 1$. Since we showed that $F^{d+1}(A)\subseteq\CH_0(A)_n$, the claim follows. 

\end{proof}

\subsection{Behavior under isogeny} In this subsection we prove \autoref{maincompareintro}. The following proposition shows that the Gazaki filtration $\{F^r(A)\}_{r\geq 0}$ behaves well under isogeny. 

\begin{prop}\label{isogeny1} Let $A,B$ be isogenous abelian varieties via an isogeny $\varphi:A\to B$. Then $\varphi_\star(F^r(A))\subset \varphi_\star(F^r(B))$, for all $r\geq 0$. That is, the push-forward map $\varphi_\star$ is a filtered homomorphism with respect to the Gazaki filtration. 
\end{prop} 
\begin{proof}
Since isogenies are proper morphisms, $\varphi$ induces a push-forward map $\varphi_\star:\CH_0(A)\to\CH_0(B)$. We will prove by induction on $r\geq 0$ that $\varphi_\star$ induces an inclusion $\varphi_\star(F^r(A))\subseteq F^r(B)$. The statement is clearly true for $r=1$. Let $r\geq 1$. We claim that there is a commutative diagram 
\[\xymatrix{ \CH_0(A)\ar[r]^{\varphi_\star}\ar[d]^{\Phi_{r,A}} & \CH_0(B)\ar[d]^{\Phi_{r,B}}\\
S_r(k;A)\ar[r]^{\varphi^{\otimes r}} & S_r(k;B).
}
\] Here the bottom map is induced by applying the isogeny $\varphi$ on each coordinate. Let $a\in A$ be a closed point and let $[a]\in\CH_0(A)$ be the induced zero-cycle. On the one hand we have, 
\[\varphi^{\otimes r}(\Phi_{r,A}([a]))=\{\varphi(a),\ldots,\varphi(a)\}_{k(a)/k}.\] On the other hand, $\varphi_\star([a])=[k(a):k(\varphi(a))][\varphi(a)]$, which yields,
\[\Phi_{r,B}(\varphi([a]))=[k(a):k(\varphi(a))]\{\varphi(a),\ldots,\varphi(a)\}_{k(\varphi(a))/k}.\] The equality of these two symbols follows by the projection formula of the Somekawa $K$-group (cf.~\cite[2.5]{Gazaki/Leal2022}). Namely, 
\begin{eqnarray*}\{\varphi(a),\ldots,\varphi(a)\}_{k(a)/k}=&&\{\res_{k(a)/k(\varphi(a))}(\varphi(a)),\ldots,\res_{k(a)/k(\varphi(a))}(\varphi(a))\}_{k(a)/k}\\=&&\{\varphi(a),\ldots,\varphi(a),N_{k(a)/k(\varphi(a))}(\res_{k(a)/k(\varphi(a)))}(\varphi(a))\}_{k(\varphi(a)/k}\\=&&[k(a):k(\varphi(a))]\{\varphi(a),\ldots,\varphi(a)\}_{k(\varphi(a))/k}.
\end{eqnarray*} By induction hypothesis the top horizontal map when restricted to $F^{r-1}(A)$ gives a homomorphism $F^{r-1}(A)\xrightarrow{\varphi_\star} F^{r-1}(B)$. By definition of the filtration $\{F^i(A)\}_{i\geq 0}$, we have $F^r(A)=\ker(\Phi_{r,A}|_{F^{r-1}(A)})$. For $z\in F^r(A)$, it follows that $\Phi_{r,B}(\varphi_\star(z))=\varphi^{\otimes r}(\Phi_{r,A}(z))=0$, which means $\varphi_\star(z)\in\ker(\Phi_{r,B}|_{F^{r-1}(B)})=F^{r}(B)$ as required. This completes the induction. 

\end{proof}

\begin{cor}\label{n*} Let $A$ be an abelian variety and $n\geq 1$ a positive integer. Let $n_\star:\CH_0(A)\to\CH_0(A)$ be the push-forward homomorphism induced by the multiplication by $n$ map on $A$. Then the induced homomorphism $n_\star:F^r(A)/F^{r+1}(A)\to F^r(A)/F^{r+1}(A)$ is given by multiplication by $n^r$. 
\end{cor}
\begin{proof} The homomorphism $n^{\otimes r}:S_r(k;A)\to S_r(k;A)$ is clearly multiplication by $n^r$, since the symbol is multilinear. Moreover, 
it follows by the proof of \autoref{isogeny1} that we have a commutative diagram 
\[\xymatrix{ F^r(A)/F^{r+1}(A)\ar[r]^{n_\star}\ar[d]^{\Phi_{r,A}} & F^r(A)/F^{r+1}(A)\ar[d]^{\Phi_{r,A}}\\
S_r(k;A)\ar[r]^{n^r} & S_r(k;A)
}
\]
with injective vertical maps. Let $z\in F^r(A)/F^{r+1}(A)$. Since $\Phi_{r,A}$ is a homomorphism, it follows that $\Phi_{r,A}(n_\star(z))=n^r(\Phi_{r,A}(z))=\Phi_{r,A}(n^rz)$. Since $\Phi_{r,A}$ is injective, we conclude that $n_\star(z)=n^rz$. 

\end{proof}

\begin{theo}\label{maincompare}  Let $A$ be a self-dual abelian variety of dimension $d$ admitting an isogeny $\varphi:A\to X=E_1\times\cdots \times E_d$ to a product of elliptic curves such that $\check{\varphi}\circ\varphi=n$, where $\check{\varphi}:X\to A$ is the dual isogeny. Then for all $r\geq 0$ we have an equality \[\displaystyle F^r(A)\otimes\Z\left[\frac{1}{2n}\right]=G^r(A)\otimes\Z\left[\frac{1}{2n}\right].\] 
% Moreover, for every $r>d$, $\displaystyle F^r(A)\otimes\Z\left[\frac{1}{2n}\right]=0$. 
\end{theo}

\begin{proof}
For every $r\geq 0$ we have inclusions $G^r(A)\hookrightarrow F^r(A)$ which are equalities for $r=0,1,2$. In order to show that 
the inclusion becomes an isomorphism after $\displaystyle\otimes\Z\left[\frac{1}{2n}\right]$ for all $r\geq 0$, it is enough to show an isomorphism $\displaystyle  F^r(A)/F^{r+1}(A)\otimes\Z\left[\frac{1}{2n}\right]=G^r(A)/G^{r+1}(A)\otimes\Z\left[\frac{1}{2n}\right]$ for $r\geq 0$. For, the claim will follow by induction and the five lemma applied to the diagram 
\[\xymatrix{ 0 \ar[r] & G^{r+1}(A)\otimes\Z\left[\frac{1}{2n}\right]\ar[r]\ar[d] & G^r(A)\otimes\Z\left[\frac{1}{2n}\right]\ar[r]\ar[d] & G^r(A)/G^{r+1}(A)\otimes\Z\left[\frac{1}{2n}\right]\ar[r]\ar[d] & 0\\
0 \ar[r] & F^{r+1}(A)\otimes\Z\left[\frac{1}{2n}\right]\ar[r] & F^r(A)\otimes\Z\left[\frac{1}{2n}\right]\ar[r] & F^r(A)/F^{r+1}(A)\otimes\Z\left[\frac{1}{2n}\right]\ar[r] & 0.
}
\]
By assumption we have $\check{\varphi}_\star\circ\varphi_\star=(\check{\varphi}\circ\varphi)_\star=n_\star$, which is given by multiplication by $n^r$ on $F^r(A)/F^{r+1}(A)$, and hence it becomes an isomorphism after $\displaystyle\otimes\Z\left[\frac{1}{2n}\right]$. We claim that this isomorphism $n_\star\otimes 1=(\check{\varphi}_\star\otimes 1)\circ (\varphi_\star\otimes 1)$ factors through the quotient $\displaystyle G^r(A)/G^{r+1}(A)\otimes\Z\left[\frac{1}{2n}\right]$. To see this we consider the homomorphism  
\[\varphi_\star\otimes 1: F^r(A)/F^{r+1}(A)\otimes\Z\left[\frac{1}{2n}\right]\to F^r(X)/F^{r+1}(X)\otimes\Z\left[\frac{1}{2n}\right].\] It follows by \autoref{ellipticcompare} that we have an equality 
\[F^r(X)/F^{r+1}(X)\otimes\Z\left[\frac{1}{2n}\right]= F^r_{\RS}(X)/F^{r+1}_{\RS}(X)\otimes\Z\left[\frac{1}{2n}\right]= G^r(X)/G^{r+1}(X)\otimes\Z\left[\frac{1}{2n}\right].\]
  Since $\check{\varphi}_\star(G^r(X))\subset G^r(A)$ (cf.~\autoref{Grisogeny}), postcomposing with the homomorphism $\check{\varphi}_\star\otimes 1$ lands in the quotient  $\displaystyle G^r(A)/G^{r+1}(A)\otimes\Z\left[\frac{1}{2n}\right]$. Thus, we obtain a factoring 
  \[n^r\otimes 1:\frac{F^r(A)}{F^{r+1}(A)}\otimes\Z\left[\frac{1}{2n}\right]\rightarrow \frac{G^r(A)}{G^{r+1}(A)}\otimes\Z\left[\frac{1}{2n}\right]\xrightarrow{\iota_r\otimes 1}\frac{F^r(A)}{F^{r+1}(A)}\otimes\Z\left[\frac{1}{2n}\right],\] where $\iota_r$ is the homomorphism induced by the inclusion $G^r(A)\hookrightarrow F^r(A)$. Since $n^r\otimes 1$ is an isomorphism, it follows that $\iota_r\otimes 1$ is surjective. To see that it's also injective consider the composition 
  \[\frac{G^r(A)}{G^{r+1}(A)}\otimes\Z\left[\frac{1}{2n}\right]\xrightarrow{\iota_r\otimes 1}\frac{F^r(A)}{F^{r+1}(A)}\otimes\Z\left[\frac{1}{2n}\right]\xrightarrow{n^r\otimes 1}\frac{F^r(A)}{F^{r+1}(A)}\otimes\Z\left[\frac{1}{2n}\right].\] We get a factoring 
  \[\frac{G^r(A)}{G^{r+1}(A)}\otimes\Z\left[\frac{1}{2n}\right]\xrightarrow{\iota_r\otimes 1}\frac{F^r(A)}{F^{r+1}(A)}\otimes\Z\left[\frac{1}{2n}\right]\rightarrow\frac{G^r(A)}{G^{r+1}(A)}\otimes\Z\left[\frac{1}{2n}\right],\] which is again given by multiplication by $n^r$, and hence this composition is an isomorphism, and hence $\iota_r\otimes 1$ is injective. 
  %This completes the proof of the first claim. 

%  We next prove that $F^r(A)\otimes 1/2$ is $n^{d+1}$-torsion for every $r>d$. It is enough to prove that the homomorphism $n^{d+1}:F^r(A)\to F^r(A)$ is zero for every $r>d+1$. We 
  
%We next prove that for every $r>d$, $\displaystyle F^r(A)\otimes\Z\left[\frac{1}{2n}\right]=0$. 

\end{proof}

\begin{rem}\label{interchange} For an abelian variety $A$ over $\mathbb{C}$, Beauville (\cite{Beauville1986}) showed a rational vanishing $G^{d+1}(A)\otimes\Q=0$. This result has been generalized by Denninger and Murre (\cite{Denninger/Murre}) to abelian varieties defined over  arbitrary fields. We believe that this vanishing could hold integrally.  Suppose that $A$ is self-dual. We suggest a possible strategy to prove such a vanishing result in this case. Consider the isomorphism $g:A(k)\xrightarrow{\simeq} \Pic^0(A)$ constructed using a principal polarization. We think of the elements of $\Pic^0(A)$ as divisors and we consider the intersection product,
\begin{eqnarray*}&&\overbrace{\Pic^0(A)\otimes\cdots\otimes\Pic^0(A)}^d\to\CH_0(A)\\
&&D_1\otimes D_2\otimes\cdots\otimes D_d\mapsto D_1\cdot D_2\cdots\cdot D_d.
\end{eqnarray*} Precomposing with $g^{\otimes d}$ yields a map $\displaystyle\overbrace{A(k)\otimes\cdots\otimes A(k)}^d\to\CH_0(A)$. Let $a_1,\cdots, a_d\in A(k)$. Each $a_i\in A(k)$ gives rise to a zero-cycle $[a_i]-[0]\in F^1(A)$. We believe that applying $g$ interchanges the intersection product with the Pontryagin product. Namely, 
 \[g(a_1)\cdot g(a_2)\cdots \cdot g(a_d)=([a_1]-[0])\odot([a_2]-[0])\odot\cdots\odot([a_d]-[0]).\]
If this formula is true, this will automatically imply $G^{d+1}(A)=0$, for otherwise the map
\begin{eqnarray*}
&&\overbrace{A(k)\times\cdots\times A(k)}^d\to G^d(A)\\
&&(a_1,\ldots,a_d)\mapsto ([a_1]-[0])\odot([a_2]-[0])\odot\cdots\odot([a_d]-[0])
\end{eqnarray*} won't be multilinear. It is likely that such a result could follow from an integral version of the Fourier Mukai transform (see \cite[Remark 4.6]{Gazaki2015}). 
\end{rem}

\vspace{2pt}
\section{Results over $p$-adic fields}\label{padicsection} Throughout this section $k$ will be a finite extension of the $p$-adic field $\Q_p$, where $p$ is an \textbf{odd} prime. For an abelian group $B$ we will denote by $B_{\dv}$ its maximal divisible subgroup. Divisible groups are injective $\Z$-modules, so we have an isomorphism $B\simeq B_{\dv}\oplus B/B_{\dv}$.
\begin{notn}\label{nd}
For an abelian group $B$ we will denote by $B_{\nd}$ the non-divisible quotient $B/B_{\dv}$, which we consider as a subgroup of $B$. 
\end{notn} \autoref{mainintro2} will follow from \cite{Gazaki/Leal2022} and the following easy lemmas. 
\begin{lem}\label{divisibilitylemmas} Let $B$ be an abelian group. 
\begin{enumerate}
\item[(a)] Let $H$ be a subgroup of $B$. Suppose that $B_{\nd}$ is either a torsion group of finite exponent, or a finite group. Then the same holds for $(B/H)_{\nd}$. 
%The statement still holds true if torsion of finite exponent is replaced with finite. 
\item[(b)] Suppose that $B$ has a subgroup $B'$ such that $B/B'$ is  torsion of finite exponent. Assume additionally that the group $B'_{\nd}$ is torsion of finite exponent. Then the same holds for $B_{\nd}$. 
\end{enumerate}
\end{lem}

\begin{proof}
We first prove (a). Let $\varepsilon: B\to B/H$ be the projection. Since the image of a divisible group under a group homomorphism is divisible, $\varepsilon$ restricts to a homomorphism $B_{\dv}\xrightarrow{\varepsilon} (B/H)_{\dv}$. We have a commutative diagram 
\[\xymatrix{ 0 \ar[r] & B_{\dv}\ar[r]\ar[d]^{\varepsilon} & B\ar[r]\ar[d]^{\varepsilon} & B_{\nd}\ar[r]\ar@{-->}[d]^{\varepsilon} & 0\\
0 \ar[r] & (B/H)_{\dv}\ar[r] & B/H\ar[r] & (B/H)_{\nd}\ar[r] & 0.
}
\] Since the middle vertical map is surjective, so is the rightmost vertical map. 
 By assumption the group $B_{\nd}$ is torsion of finite exponent. Surjectivity implies that the same holds for $(B/H)_{\nd}$. The same argument applies if $B_{\nd}$ is finite. 
 
 We next prove (b). Suppose that $B/B'$ is torsion of exponent $n\geq 1$. Moreover, let $N\geq 1$ be such that $NB'_{\nd}=0$. Let $b\in B$. It follows that $nb\in B'$. Moreover, $Nnb\in B'_{\dv}$. Since $B'_{\dv}\leq B_{\dv}$, it follows that $nN B_{\nd}=0$. 

\end{proof}

\begin{lem}\label{basechange} Let $X$ be a smooth projective variety over $k$. Suppose there exists a finite extension $K/k$ such that $F^2(X_K)$ is the direct sum of its maximal divisible subgroup with a torsion group of finite exponent. Then the same holds for $F^2(X)$. 
\end{lem}

\begin{proof} Consider the push-forward, $\CH_0(X_K)\xrightarrow{\pi_{K/k\star}}\CH_0(X)$, and the pull-back, $\CH_0(X)\xrightarrow{\pi_{K/k}^\star}\CH_0(X_K)$ induced by the proper and flat projection $\Spec(K)\to\Spec(k)$. These satisfy the equality $\pi_{K/k\star}\circ\pi_{K/k}^\star=[K:k]$. We have a commutative diagram 
%\[\begin{tikzcd} 0 \ar{r} & F^2(X)\ar{r}\ar{d} & F^1(X)\ar{r}\ar{d} & \Alb_X(k)\ar{r}\ar{d} & 0\\
%0 \ar{r} & F^2(X_K)\ar{r} & F^1(X_K)\ar{r} & \Alb_{X_K}(K)\ar{r} & 0
%\end{tikzcd}\]

\[\xymatrix{ 0 \ar[r] & F^2(X)\ar[r]\ar@{-->}[d]^{\pi_{K/k}^\star} & F^1(X)\ar[r]^{\alb_X}\ar[d]^{\pi_{K/k}^\star} & \Alb_X(k)\ar[r]\ar[d] & 0\\
0 \ar[r] & F^2(X_K)\ar[r] & F^1(X_K)\ar[r]^{\alb_{X_K}} & \Alb_{X_K}(K)\ar[r] & 0,
}
\] and a similar diagram with the vertical arrows reversed for $\pi_{K/k\star}$. Here the rightmost vertical map is induced by the universal property of the Albanese variety. Thus, we obtain homomorphisms $F^2(X_K)\xrightarrow{\pi_{K/k\star}}F^2(X)$ and $F^2(X)\xrightarrow{\pi_{K/k}^\star}F^2(X_K)$. Let \[H=\img(F^2(X_K)\xrightarrow{\pi_{K/k\star}}F^2(X))\leq F^2(X).\] Then $F^2(X)/H$ is $[K:k]$-torsion.
% It follows that the subgroup $C=\img(\pi_{K/k\star})$ of $F^2(X)$ is torsion of exponent dividing $[K:k]$. 
Since by assumption $F^2(X_K)_{\nd}$ is torsion of finite exponent, it follows by \autoref{divisibilitylemmas} (a) that so is the group $H_{\nd}$. Then it follows by \autoref{divisibilitylemmas} (b) that $F^2(X)_{\nd}$ is torsion of finite exponent.

\end{proof}

We next prove \autoref{mainintro2} which we restate below.

\begin{theo}\label{main2} Let $X=C_1\times\cdots\times C_d$ be a product of smooth projective curves over the $p$-adic field $k$.  Let $J_i$ be the Jacobian variety of $C_i$ for $i=1,\ldots,d$. Suppose that there exists a finite extension $K/k$ over which for each $i\in\{1,\ldots,d\}$ there exists an isogeny $\psi_i:J_i\otimes_k K\to E_{i,1}\times\cdots\times E_{i,r_i}$  to a product $Y_i:=E_{i,1}\times\cdots\times E_{i,r_i}$ of elliptic curves all of which have either good reduction or split multiplicative reduction. Suppose that for each $i\in\{1,\ldots,d\}$, $\check{\psi}_i\circ\psi_i=n_i$ is an integer coprime to $p$. 
\begin{enumerate}
\item Suppose that for at most one $i\in\{1,\ldots,d\}$ there exist indexes $l$ with $l\in\{1,\ldots, r_i\}$ such that the elliptic curve $E_{i,l}$ has good supersingular reduction.  
%Suppose that for at most one $i\in\{1,\ldots,d\}$ the abelian variety $Y_i$ has coordinates with good supersingular reduction. Then the group $F^2(X)_{\nd}$ is torsion of finite exponent and there exists a finite extension $K'/K$ such that $F^2(X_{K'})_{\nd}$ is a finite group. 
\item Under the same assumptions as in part (1), if $K=k$ is an unramified extension of $\Q_p$ and all the elliptic curves involved have good reduction, then $F^2(X)$ is $p$-divisible. 
\item Suppose that each $Y_i$ has potentially good ordinary reduction. Then there exists a finite extension $L/k$ such that the cycle class map 
\[\CH_0(X_L)/p^n\xrightarrow{c_{p^n}} H^{2d}_{\text{\'{e}t}}(X_L,\mu_{p^n}^{\otimes d})\] to \'{e}tale cohomology is injective for every $n\geq 1$. 
\end{enumerate}
\end{theo}
\begin{proof}
 For statements (1), (3) we are allowed to extend to a finite extension. Thus, we may assume that the isogenies $\psi_i$ are defined over $k$, all elliptic curves $E_{i,j}$ have either good or split multiplicative reduction, and $C_i(k)\neq\emptyset$ for $i=1,\ldots, d$. It then follows by \cite{Raskind/Spiess2000} that there is an isomorphism
\[F^2(X)\simeq\bigoplus_{\nu=2}^d\bigoplus_{1\leq i_1<i_2<\cdots<i_\nu}K(k;J_{i_1},\ldots, J_{i_\nu}),\] where $K(k;J_{i_1},\ldots, J_{i_\nu})$ is the Somekawa $K$-group attached to $J_{i_1},\ldots, J_{i_\nu}$.

We first prove (1). It follows by \cite[Theorem 3.5]{Raskind/Spiess2000} that each group $K(k;J_{i_1},\ldots, J_{i_\nu})$ is the direct sum of a finite group and a group divisible by any integer $m$ coprime to $p$. Thus, the same holds for $F^2(X)$. It remains to show that the group $\varprojlim\limits_n K(k;J_{i_1},\ldots, J_{i_\nu})/p^n$ is torsion of finite exponent (see \cite[Lemma 3.4.4]{Raskind/Spiess2000} for a proof of this claim). 
By enlarging the base field if necessary we make the following assumptions: $E_{i,j}[p]\subset E_{i,j}(k)$ for all $i,j$, and $\mu_{p^2}\subset k^\times$. Under these assumptions we will in fact show that the group $\varprojlim\limits_n K(k;J_{i_1},\ldots, J_{i_\nu})/p^n$ is  finite. 
For every $n\geq 1$ we have a surjection 
\[(J_{i_1}\otimes^M\ldots\otimes^M J_{i_\nu})(k)/p^n\twoheadrightarrow K(k;J_{i_1},\ldots, J_{i_\nu})/p^n.\] Thus, if we show $(J_{i_1}\otimes^M\ldots\otimes^M J_{i_\nu})(k)/p^n=(J_{i_1}\otimes^M\ldots\otimes^M J_{i_\nu})(k)/p^m$ for $n,m$ large enough, the same will be true for $K(k;J_{i_1},\ldots, J_{i_\nu})$. In particular, if $\varprojlim\limits_n (J_{i_1}\otimes^M\ldots\otimes^M J_{i_\nu})(k)/p^n$  is finite, then $\varprojlim\limits_n K(k;J_{i_1},\ldots, J_{i_\nu})/p^n$ is  finite. 

%Since the $K$-group  $K(k;J_{i_1},\ldots, J_{i_\nu})$ is a quotient of the Mackey product $(J_{i_1}\otimes^M\ldots\otimes^M J_{i_\nu})(k)$, it suffices to show that 
%$\varprojlim\limits_n (J_{i_1}\otimes^M\ldots\otimes^M J_{i_\nu})(k)/p^n$  is finite. 

We fix an index $i\in\{1,\ldots,d\}$. Let $L/k$ be an arbitrary finite extension. The isogeny $\psi_i: J_i\to E_{i,1}\times\cdots\times E_{i,r_i}$ induces an isomorphism 
\[J_i(L)/p^n\xrightarrow{\simeq} E_{i,1}(L)/p^n\oplus\cdots\oplus E_{i,r_i}(L)/p^n.\] For, we have equalities $\check{\psi}_i\circ\psi_i=n_i$ which are assumed to be coprime to $p$. Moreover, notice that the above isomorphism commutes with norm and restrictions maps. Namely, if $L/K/k$ is a tower of finite extensions, then the diagrams 
\[\begin{tikzcd}
	J_{i}(L)/p^n\ar{r}{\psi_i}\ar{d}{N_{L/K}}  & \bigoplus_{j=1}^{r_i} E_{i,j}(L)/p^n\ar{d}{N_{L/K}} \\
	J_{i}(K)/p^n\ar{r}{\psi_i} & \bigoplus_{j=1}^{r_i} E_{i,j}(K)/p^n
	\end{tikzcd},\] and 
	\[\begin{tikzcd}
	J_{i}(L)/p^n\ar{r}{\psi_i}  & \bigoplus_{j=1}^{r_i} E_{i,j}(L)/p^n \\
	J_{i}(K)/p^n\ar{r}{\psi_i} \ar{u}{\res_{L/K}} & \bigoplus_{j=1}^{r_i} E_{i,j}(K)/p^n \ar{u}{\res_{L/K}}
	\end{tikzcd}\] commute.
	 This means that we have an isomorphism of Mackey functors $J_i/p^n\simeq \bigoplus_{j=1}^{r_i} E_{i,j}/p^n$ for every $i\in\{1,\ldots,d\}$. Since $\otimes^M$ is a product structure in the abelian category of Mackey functors (cf.~\cite{Kahn1992}), the group $(J_{i_1}\otimes^M\ldots\otimes^M J_{i_\nu})(k)/p^n$ is a finite direct sum of groups of the form $(E_{i_1, l_1}\otimes^M\cdots\otimes^M E_{i_\nu, l_\nu})(k)/p^n$
with the indexes $l_j$ running in the set $\{1,\ldots,r_{i_j}\}$ for $j=1,\ldots,\nu$. Thus, it suffices to show each $\varprojlim\limits_n(E_{i_1, l_1}\otimes^M\cdots\otimes^M E_{i_\nu, l_\nu})(k)/p^n$ is finite. By our assumption, if there exist curves $E_{i_j, l_j}$, $E_{i_k, l_k}$ with good supersingular reduction, then $i_j=i_k$. In particular, in each Mackey product $(E_{i_1, l_1}\otimes^M\cdots\otimes^M E_{i_\nu, l_\nu})(k)$ there is at most one elliptic curve with good supersingular reduction.  Because of this, the finiteness of $\varprojlim\limits_n (E_{i_1, l_1}\otimes^M\ldots\otimes^M E_{i_\nu, l_\nu})(k)/p^n$ when $\nu=2$ follows by \cite[Theorem 3.24]{Gazaki/Leal2022} and its proof. For $\nu\geq 3$ this follows by \cite[Corollary 3.28]{Gazaki/Leal2022}. The rest of statement (1) follows by \autoref{basechange}. 

To show (2) we assume that $k=K$ is an unramified extension of $\Q_p$. It follows by \cite[Theorem 1.4, Corollary 1.5]{Gazaki/Hiranouchi2021} that the group $\varprojlim\limits_n(E_{i_1, l_1}\otimes^M\cdots\otimes^M E_{i_\nu, l_\nu})(k)/p^n$ is zero, which implies that $F^2(X)$ is $p$-divisible. 

Lastly, we show (3). It follows by \cite[Proposition 2.4]{Yamazaki2005}
that the injectivity of the cycle map $\CH_0(X)/p^n\xrightarrow{c_{p^n}} H^{2d}_{\text{\'{e}t}}(X,\mu_{p^n}^{\otimes d})$ can be reduced to proving injectivity of the Galois symbols
\[K(k;J_{i_1},\ldots,J_{i_\nu})/p^n\xrightarrow{s_{p^n}} H^{2\nu}(k,J_{i_1}[p^n]\otimes\cdots\otimes J_{i_\nu}[p^n]),\] for every $1\leq i_1<i_2<\cdots<i_\nu\leq d$. For a definition of the Galois symbol $s_{p^n}$ we refer to \cite[Section 2.1]{Gazaki/Leal2022}. The isogenies $\psi_i: J_i\to Y_i$ are defined over $k$ and for each $i\in\{1,\ldots,d\}$ it follows that $\ker(\psi_i)\subset J_i[m_i]$ for some positive integer $m_i$ coprime to $p$. Thus, $\psi_i$ induces an isomorphism of $\Gal(\overline{k}/k)$-modules $\psi_i:J_i[p^n]\xrightarrow{\simeq} E_{i,1}[p^n]\oplus\cdots\oplus E_{i,r_i}[p^n]$ for every $n\geq 1$.  
Using the analysis we did in part (1) we conclude that injectivity of $c_{p^n}$ can be reduced to injectivity of the Galois symbols 
\[K(k;E_{i_1,l_1},\ldots, E_{i_\nu, l_\nu})/p^n\xrightarrow{s_{p^n}}H^{2\nu}(k,E_{i_1,l_1}[p^n]\otimes\cdots\otimes E_{i_\nu, l_\nu}[p^n]),\] for all $1\leq i_1<i_2<\cdots<i_\nu\leq d$ and $l_j$ running in the set $\{1,\ldots,r_{i_j}\}$ for $j=1,\ldots,\nu$. When $\nu\geq 3$, it follows by \cite[Corollary 3.28]{Gazaki/Leal2022} that there exists a finite extension $K/k$ such that the group $K(K;E_{i_1,l_1},\ldots, E_{i_\nu, l_\nu})$ is $p$-divisible, which forces the corresponding Galois symbol to be injective. The case $\nu=2$ follows by \cite[Theorem 3.14, Corollary 3.16]{Gazaki/Leal2022}.

\end{proof}

\begin{cor} Consider the set-up of \autoref{main2} and suppose that $d=2$, so that $X=C_1\times C_2$ is a surface. Then $F^2(X)_{\nd}$ is finite unconditionally on the base field. 
\end{cor}
\begin{proof}
It follows by \autoref{main2} (1) that $F^2(X)_{\nd}$ is torsion of finite exponent. Let $N\geq 1$ be an integer such that $NF^2(X)_{\nd}=0$. It follows that $F^2(X)_{\nd}$ is a subgroup of $\CH_0(X)_N$, which is finite by \cite[Theorem 8.1]{CT93}.  

\end{proof}

%\begin{rem} \autoref{main2} (2)  Moreover, when $\dim(X)=2$ it follows that $F^2(X)_{\nd}$ is finite without having to pass to a finite extension. This follows similary to the proof of \autoref{main1}. 
%\end{rem} 
\begin{rem} Part (2) of \autoref{main2} verifies a conjecture proposed by the author and T. Hiranouchi in \cite[Conjecture 1.3]{Gazaki/Hiranouchi2021}, motivated by local-to-global considerations. The proof can be easily extended to the case when $K/k/\Q_p$ is a tower of unramified extensions with $[K:k]$ coprime to $p$. When all elliptic curves involved in the isogenies have good reduction, the groups $K(k;E_{i_1,l_1},\ldots, E_{i_\nu, l_\nu})$ are $m$-divisible for every integer $m$ coprime to $p$ (cf.~\cite[Corollary 3.5.1]{Raskind/Spiess2000}). Thus, we in fact obtain a vanishing $F^2(C_1\times\cdots\times C_d)_{\nd}=0$. 
\end{rem}
In \cite[Theorem 2.5]{Gazaki2019} the author proved that for an abelian variety $A$ over a $p$-adic field $k$ with good ordinary reduction the group $F^2(A)_{\nd}$ is torsion. In this article, using some of the computations in \autoref{comparesection}, we extend this result to the case of an abelian variety $A$ isogenous to a product of elliptic curves allowing one of the curves to have good supersingular reduction.  
\begin{prop}\label{mainab}
Let $A$ be a self-dual abelian variety 
 over a $p$-adic field $k$ admitting an isogeny $\varphi:A\to E_1\times\cdots\times E_d$ to a product of elliptic curves. Suppose that $\check{\varphi}\circ\varphi=n$ is an integer coprime to $p$. Assume further that each $E_i$ has good reduction and at most one $E_i$ has good supersingular reduction. Then $F^2(A)_{\nd}$ is torsion. 
 \end{prop}
 \begin{proof}
 The proof of \cite[Theorem 2.5]{Gazaki2019} will carry through verbatim if we show that for each $r\geq 3$ the Somekawa $K$-group $S_r(k;A)$ is divisible and $S_2(k;A)_{\nd}$ is finite.  First, it follows by \cite[Theorem 3.5]{Raskind/Spiess2000} that  for all $r\geq 2$ the group $S_r(k;A)$ is the direct sum of a finite group and a group divisible by any integer $m$ coprime to $p$. It remains to show that $\varprojlim\limits_n S_r(k;A)/p^n$  is finite. First notice that the group $\varprojlim\limits_n S_r(k;E_1\times\cdots\times E_d)/p^n$ is finite. For, it follows by \autoref{Sr} that  there is an isomorphism
 \[S_r(k;E_1\times\cdots\times E_d)\xrightarrow{\simeq}\bigoplus_{1\leq i_1\leq\cdots\leq i_r\leq d}S_r(k;E_{i_1},\ldots, E_{i_r}).\] Since each $E_i$ has good reduction and $p$ is odd, it follows that $S_r(k;E_{i_1},\ldots, E_{i_r})=0$ whenever two of the indexes are the same (since it's $2$-torsion and $2$-divisible). Thus, \[S_r(k;E_1\times\cdots\times E_d)\xrightarrow{\simeq}\bigoplus_{1\leq i_1<\cdots< i_r\leq d}K(k;E_{i_1},\ldots, E_{i_r}).\] In each of the $K$-groups $K(k;E_{i_1},\ldots, E_{i_r})$ at most one of the elliptic curves has good supersingular reduction. Thus, each $\varprojlim\limits_n K(k;E_1,\cdots, E_d)/p^n$ is finite by \cite[Theorem 3.14, Corollary 3.16]{Gazaki/Leal2022}, and hence so is $\varprojlim\limits_n S_r(k;E_1\times\cdots\times E_d)/p^n$.  The isogeny $\varphi$ induces a homomorphism for all $n\geq 1$,
 \[S_r(k;A)/p^n\xrightarrow{\varphi^{\otimes r}}S_r(k;E_1\times\cdots\times E_d)/p^n.\] Since $\check{\varphi}^{\otimes r}\circ\varphi^{\otimes r}$ is multiplication by $n^r$, which is coprime to $p$, and so is $\varphi^{\otimes r}\circ\check{\varphi}^{\otimes r}$, this map is an isomorphism for all  $n\geq 1$, and hence the claim follows. 
 %In fact, each of the Mackey products $\varprojlim\limits_n (E_1\otimes^M\cdots\otimes^M E_d)(k)/p^n$ is finite. 

 \end{proof}
 
 \begin{rem}\label{finiteness} If an integral vanishing $G^{d+1}(A)=0$ is indeed true (cf.~\autoref{interchange}), this will automatically improve \autoref{mainab} and give new infinite classes of abelian varieties for which the full \autoref{mainconj} holds. Namely, \autoref{maincompare} gives an equality 
\[\displaystyle F^{d+1}(A)\otimes\Z\left[\frac{1}{2n}\right]=G^{d+1}(A)\otimes\Z\left[\frac{1}{2n}\right]=0.\] Let $l_1,\ldots,l_m$ be the prime numbers that appear in the prime factorization of $2n$. Then the group  $F^{d+1}(A)$ is torsion and in particular,
\[F^{d+1}(A)=F^{d+1}(A)\{l_1\}\oplus\cdots\oplus F^{d+1}(A)\{l_m\}.\] Since we assumed that $p$ is odd and coprime to $n$, it follows that each of the subgroups $F^{d+1}(A)\{l_j\}$, for $j=1,\ldots,m$, is $p$-divisible, and hence so is $F^{d+1}(A)$. Assume we are in the set-up of \autoref{mainab} and suppose for simplicity that $k$ is an unramified extension of $\Q_p$. Then it follows by \cite[Theorem 1.4, Corollary 1.5]{Gazaki/Hiranouchi2021} that each of the quotients $F^r(A)/F^{r+1}(A)\simeq S_r(k;A)$ is $p$-divisible for $r\geq 2$. We conclude that $F^2(A)$ is $p$-divisible. Combined  with \cite[Theorem 3.5]{Raskind/Spiess2000}, this shows that \autoref{mainconj} will be true in this case. 
 \end{rem}

\vspace{3pt}
\subsection{Explicit new evidence for \autoref{mainconj}}\label{computations} Let $E_1, E_2$ be two elliptic curves over a perfect field $k$ such that $E_i[2]\subset E_i(k)$ for $i=1,2$. The curve $E_1$ is isomorphic to a curve $E_{a,b}$ given by a Weierstrass equation of the form 
$E_{a,b}:y^2=x(x-a)(x-b)$ with $a,b\in k, a\neq b$. Similarly, $E_2$ is isomorphic to a curve $E_{c,d}:y^2=x(x-c)(x-d)$ for some $c,d\in k$, $c\neq d$. We consider the curve $C_{a,b,c,d}$ constructed in \cite{Scholten} given by the equation 
\[C_{a,b,c,d}:(ad-bc)y^2=((a-b)x^2-(c-d))(ax^2-c)(bx^2-d),\] which admits dominant morphisms $h_i:C_{a,b,c,d}\to E_i$ for $i=1,2$. The isomorphisms \[E_{a,b}\simeq E_{b,a}\simeq E_{-b,b-a}\simeq E_{b-a,-a}\simeq E_{a-b,-b}\simeq E_{-b,a-b}\] produce typically six non-isomorphic curves 
%$C_{a,b}, C_{b,a}, C_{-b,b-a},C_{b-a,-a},C_{a-b,-b}, C_{-b,a-b}$
 which in the very general case are smooth of genus $2$ (cf.~\cite[Section 4.2.2]{GazakiLove2022}). Notice that even when $E_1=E_2$, some of the six curves are smooth. Thus, for any given pair $E_1,E_2$ as above we can find up to six non-isomorphic hyperelliptic curves of genus $2$ whose Jacobian $J_C$ is isogenous over $k$ to the product $E_1\times E_2$.
 
  One can use this construction and  \autoref{main2} to obtain infinitely many new examples of surfaces satisfying \autoref{mainconj}. We consider products $C\times E$, where $E$ is any other elliptic curve. We can easily choose $E_1,E_2, E$ in such a way so that the reduction assumption of \autoref{main2} is satisfied. 
We note that for an elliptic curve $E$ over a number field $F$ without potential complex multiplication a result of Serre (cf.~\cite{Serre1981}) shows that the set of places of $F$ at which $E$ has good supersingular reduction is a set of places of density zero. In the case $F=\Q$ it is actually an infinite set (cf.~\cite{Elkies1987})), but it is not known if this is also the case for an arbitrary number field. On the other hand, if $E$ has potential CM, the set of supersingular primes is a set of density $1/2$.

% Let $\iota:C\to J$ be the 

\vspace{15pt}

\bibliographystyle{amsalpha}

\bibliography{bibfile}

\end{document}